\documentclass[letterpaper]{amsart}

\usepackage{hyperref}

\usepackage{amssymb}
\usepackage{amsmath}
\usepackage{amsthm}
\usepackage{enumerate}
\usepackage{graphicx}
\usepackage{caption}

\theoremstyle{plain}
\newtheorem{thm}{Theorem}[section]

\newtheorem{lem}[thm]{Lemma}

\newtheorem{obs}[thm]{Observation}

\theoremstyle{definition}

\theoremstyle{remark}
\newtheorem*{remark}{Remark}

\newtheorem*{acknowledgements}{Acknowledgements}

\newcommand{\what}{\widehat}
\DeclareMathOperator{\re}{Re}

\begin{document}

\title{The converse of the Schwarz Lemma is false}
\author{Maxime Fortier Bourque}
\address{Department of Mathematical and Computational Sciences, University of Toronto Mississauga, 3359 Mississauga Road, Mississauga, ON, L5L 1C6, Canada}
\email{maxforbou@gmail.com}


\begin{abstract}
Let $h:X \to Y$ be a homeomorphism between hyperbolic surfaces with finite topology. If $h$ is homotopic to a holomorphic map, then every closed geodesic in $X$ is at least as long as the corresponding geodesic in $Y$, by the Schwarz Lemma. The converse holds trivially when $X$ and $Y$ are disks or annuli, and it holds when $X$ and $Y$ are closed surfaces by a theorem of W. Thurston. We prove that the converse is false in all other cases, strengthening a result of Masumoto.
\end{abstract}

\maketitle

\section{Introduction}

In this paper, a \emph{hyperbolic surface} is a Riemann surface (a connected $1$-dimensio\-nal complex manifold without boundary) with finitely generated fundamental group and whose universal cover is biholomorphic to the unit disk. Every hyperbolic surface is equipped with its complete Poincar\'e metric coming from the unit disk. Let $h:X \to Y$ be a homeomorphism between two hyperbolic surfaces. We are interested in the relationship between the following conditions:

\begin{enumerate}[(a)]
\item $h$ is homotopic to a conformal embedding;
\item $h$ is homotopic to a holomorphic immersion;
\item $h$ is homotopic to a holomorphic map;
\item $h$ does not increase the length of any homotopy class of closed curve;
\item $h$ does not increase the length of any homotopy class of simple closed curve.
\end{enumerate}

By condition (d) we mean that for every closed curve $\alpha$ in $X$ we have the inequa\-lity  $\ell_Y([h(\alpha)])\leq \ell_X([\alpha])$. Here $[\beta]$ is the free homotopy class of the closed curve $\beta$ and $\ell_Z([\beta])$ is the infimum of the lengths of curves $\gamma$ in the set $[\beta]$, where length is measured in the Poincar\'e metric on $Z$. The quantity $\ell_Z([\beta])$ is equal to zero if $\beta$ can be deformed into a point or a cusp, and is equal to the length of the unique geodesic freely homotopic to $\beta$ otherwise. Condition (e) is the same inequality, but restricted to homotopy classes of \emph{simple} (i.e. embedded) closed curves.

\begin{obs}
The implications $(a) \Rightarrow (b) \Rightarrow (c) \Rightarrow (d) \Rightarrow (e)$ hold.
\end{obs}

The implications (a) $\Rightarrow$ (b) $\Rightarrow$ (c) and (d) $\Rightarrow$ (e) are trivial, and the implication (c) $\Rightarrow$ (d) is a consequence of the Schwarz Lemma, which says that holomorphic maps between hyperbolic surfaces are $1$-Lipschitz.

If $X$ and $Y$ are disks or annuli, then (e) $\Rightarrow$ (a). Indeed, up to biholomorphism there is only one hyperbolic surface homeomorphic to a disk. Moreover, any hyperbolic surface homeomorphic to an annulus is biholomorphic to $S^1 \times (0,m)$ for some $m \in (0,\infty]$. There is only one homotopy class of unoriented simple closed curve in the annulus $S^1 \times (0,m)$ and its length is $\pi / m$. If $\pi / m_2 \leq \pi / m_1$, then $m_1 \leq m_2$ and the inclusion map from $S^1 \times (0,m_1)$ to $S^1 \times (0,m_2)$ provides a conformal embedding homotopic to a homeomorphism, which shows that (e) $\Rightarrow$ (a) for annuli.

The implication (e) $\Rightarrow$ (a) also holds when $X$ and $Y$ are closed surfaces or when both have finite area by a theorem of W. Thurston \cite[Theorem 3.1]{Thurston}. 

\begin{thm}[W. Thurston]
Let $h:X \to Y$ be a homeomorphism between hyperbolic surfaces of finite area. If $h$ does not increase the length of any homotopy class of simple closed curve, then $h$ is homotopic to an isometry.
\end{thm}

In particular, if all the simple closed geodesics in $Y$ are at most as long as the corresponding geodesics in $X$, then they have the same length as the ones in $X$. This is false for hyperbolic surfaces with infinite area \cite{Parlier,Papadopoulos,Gendulphe}.

In his paper, W. Thurston proves the implication (e) $\Rightarrow$ (d) directly, and his proof applies to surfaces of infinite area as well \cite[Proposition 3.5]{Thurston}.

\begin{thm}[W. Thurston]
Let $h:X \to Y$ be a homeomorphism between hyperbolic surfaces. If $h$ does not increase the length of any homotopy class of simple closed curve, then $h$ does not increase the length of any homotopy class of closed curve.
\end{thm}

In \cite{Masumoto}, Masumoto shows that the implication (c) $\Rightarrow$ (a) is false in general for surfaces of infinite area.

\begin{thm}[Masumoto]
Let $X$ be a hyperbolic surface of positive genus and infinite area. Then there exist a hyperbolic surface $Y$ and a homeomorphism $h:X \to Y$ which is homotopic to a holomorphic map but is not homotopic to a conformal embedding.
\end{thm}

As Masumoto observes, it follows that (e) does not imply (a) since (c) implies (e). It turns out that the right invariant for conformal embeddings is not hyperbolic length but extremal length. Indeed, condition (a) holds if and only if $h$ does not increase the extremal length of any homotopy class of weighted multicurve \cite{Dylan}. This was first proved by Kerckhoff for closed surfaces \cite[Theorem 4]{Kerckhoff} and by Masumoto for one holed tori \cite[Theorem 5.1]{Masumoto2}. In both of these situations it suffices to compare the extremal length of homotopy classes of simple closed curves, but this is not the case in general.

The goal of this paper is to show that the implications (b) $\Rightarrow$ (a) and (d) $\Rightarrow$ (c) are false whenever $X$ and $Y$ are neither disks, annuli, or closed surfaces. In particular, our method handles surfaces of genus zero, which answers a question of Masumoto. Our results are complementary to Masumoto's: we start with the codomain $Y$ and construct a domain $X$ with the desired properties. 

The first theorem addresses the relationship between conformal embeddings and holomorphic immersions.

\begin{thm}\label{thm:immbutnotemb}
Let $Y$ be a hyperbolic surface which is neither a disk, an annulus, a closed surface, nor the triply punctured sphere. Then there exist a hyperbolic surface $X$ all of whose ends are funnels and a homeomorphism $h : X \to Y$ which is homotopic to a holomorphic immersion but is not homotopic to a conformal embedding.
\end{thm}

\begin{remark}
By gluing a punctured disk to each ideal boundary component of a pair of pants one obtains the triply punctured sphere. Thus every pair of pants embeds conformally in the triply punctured sphere. However, if $Y$ is any pair of pants with infinite area, then the above theorem says that we can find another pair of pants $X$ with no cusp which immerses holomorphically in $Y$ but does not embed conformally.
\end{remark}

We use Theorem \ref{thm:immbutnotemb} to show that the converse of the implication (c) $\Rightarrow$ (d) (which is a consequence of the Schwarz Lemma) is false.

\begin{thm}\label{thm:noconverse}
Let $Y$ be a hyperbolic surface which is neither a disk, an annulus, a closed surface, nor the triply punctured sphere. Then there exist a hyperbolic surface $X$ all of whose ends are funnels, a homeomorphism $h : X \to Y$, and a constant $0<r < 1$ such that $\ell_Y([h(\alpha)]) \leq r \cdot \ell_X([\alpha])$ for every closed curve $\alpha$ in $X$ but such that $h$ is not homotopic to a holomorphic map.
\end{thm}

We do not know whether (c) implies (b).

\section{The Schwarz Lemma and Montel's Theorem} \label{sec:schwarz}

We first recall the invariant version of the Schwarz Lemma \cite[p.22]{Milnor}.

\begin{thm}[Schwarz--Pick]
If $f:X \to Y$ is a holomorphic map between hyperbolic surfaces, then $\|d_xf\| \leq 1$ for every $x \in X$, where the norm of the derivative is taken with respect to the Poincar\'e metrics on $X$ and $Y$. If equality holds at a single point, then $f$ is a covering map and hence a local isometry.
\end{thm}

Equivalently, holomorphic maps between hyperbolic surfaces do not increase distances. In particular, they do not increase the lengths of homotopy classes of closed curves. 

Another consequence is that holomorphic maps are equicontinuous. This leads to a compactness result known as Montel's Theorem \cite[p.34]{Milnor}. We say that a sequence of maps $(f_n)_{n=1}^\infty$ between two topological spaces $X$ and $Y$ is \emph{compactly divergent} if for every compact sets $K \subset X$ and $L \subset Y$, there exists an $n_0 \in \mathbb{N}$ such that $f_n(K)\cap L = \varnothing$ for all $n\geq n_0$. A family $\mathcal{F}$ of holomorphic maps between Riemann surfaces is \emph{normal} if every sequence in $\mathcal F$ contains either a subsequence converging to a holomorphic map or a compactly divergent subsequence.

\begin{thm}[Montel]
The set of all holomorphic maps between two hyperbolic surfaces is normal.
\end{thm}

\section{Immersions not homotopic to embeddings}

A \emph{slit complement} in a hyperbolic surface $Y$ is an open subset $Z \subset Y$ whose complement is a finite union of analytic arcs such that there exists a non-zero integrable holomorphic quadratic differential on $Y$ which extends analytically to the ideal boundary of $Y$ and is non-negative along that ideal boundary as well as along the arcs in $Y \setminus Z$. An arc along which a quadratic differential $q$ is non-negative is said to be \emph{horizontal} for $q$. A good reference for the definition and basic properties of quadratic differentials is \cite{Strebel}.

We will need the fact that slit complements cannot be deformed too much by conformal embeddings \cite[Theorem 3.1]{Ioffe}.

\begin{thm}[Ioffe]
Let $Z$ be a slit complement for a quadratic differential $q$ on a hyperbolic surface $Y$. If $f: Z \to Y$ is a conformal embedding homotopic to the inclusion map $Z \hookrightarrow Y$, then $f$ is an isometry with respect to the metric $|q|^{1/2}$ and $f(Z)$ is a slit complement for $q$.
\end{thm}

In his paper, Ioffe states that $f$ has to be equal to the inclusion map, but this is false \cite{couch}. His proof of the above weaker statement is nonetheless correct. We use this theorem to construct holomorphic immersions homotopic to homeomorphisms but not homotopic to conformal embeddings.

\begin{proof}[Proof of Theorem \ref{thm:immbutnotemb}]
First assume that $Y$ is not a pair of pants, and let $\gamma$ be a non-trivial simple closed curve in $Y$ which is not boundary parallel. Let $q$ be the Jenkins-Strebel quadratic differential on $Y$ whose closed horizontal trajectories have full measure in $Y$ and are all homotopic to $\gamma$ \cite[Chapter VI, \S 21]{Strebel}. Observe that the ideal boundary of $Y$ is horizontal for $q$ and that $q$ has at least one zero on each ideal boundary component. Indeed, if $q$ has no zero on an ideal boundary component, then it has a non-empty annulus of closed horizontal trajectories parallel to that boundary component. 

For every cusp of $Y$, pick a horizontal arc $[0,1) \to Y$ converging to the cusp. Do this in such a way that the arcs are pairwise disjoint. Note that there is indeed at least one horizontal trajectory emanating from every cusp, since $q$ extends meromorphically at the cusps with at most simple poles there. If $Y$ has no cusp, then pick a horizontal arc $[0,1) \to Y$ converging to a zero of $q$ on the ideal boundary. In either case, remove the chosen arcs from $Y$ and call the resulting slit complement $Z$. Then take a half-disk $D \approx \{z \in \mathbb{C} : |z|<1, \re z\geq 0 \}$ lying on the right side of one of the chosen arcs and glue a copy $D'$ of $D$ to the left side of the slit in $Z$. The resulting Riemann surface $X=Z \sqcup D'$ comes with a holomorphic immersion $f:X \to Y$ which is the inclusion map on $Z$ and is the forgetful map identifying $D'$ with $D\subset Y$ elsewhere.

\begin{figure}[htp]
\includegraphics[scale=.8]{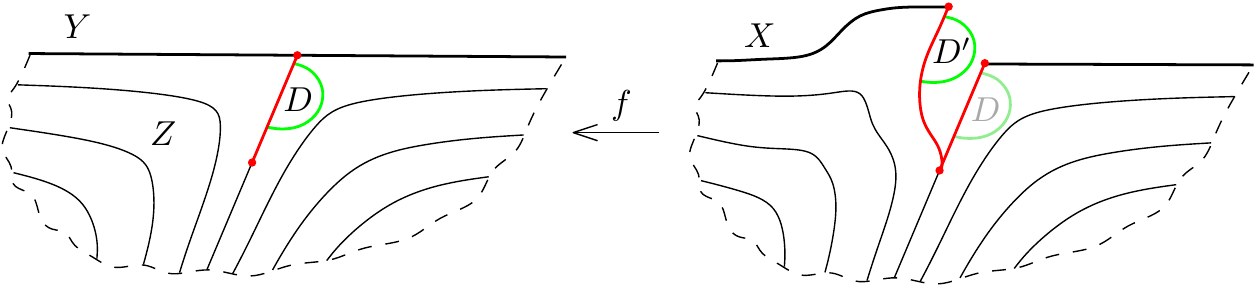}
\caption{The construction of the surface $X$ and the immersion $f$. In this picture, the arc converges to a zero of the quadratic differential on an ideal boundary component of $Y$. One may think of $D'$ as lying above $D$ and think of $f$ as the downwards projection.} \label{fig:slitmap}
\end{figure} 

The map $f$ is homotopic to a homeomorphism and by construction no end of $X$ is a cusp. Suppose there is a conformal embedding $g: X \to Y$ homotopic to $f$. Then the restriction of $g$ to the slit complement $Z\subset Y$ is a conformal embedding and the complement of $g(Z)$ contains the open set $g(D')$, which contradicts Ioffe's Theorem. Thus no such $g$ exists.

If $Y$ is a pair of pants but not the triply punctured sphere, then at least one of its ends $e$ is a funnel. In this case, let $q$ be the Jenkins-Strebel quadratic differential on $Y$ all of whose closed horizontal trajectories are parallel to $e$, cut $Y$ along two disjoint horizontal arcs converging to the other two ends, and repeat the above construction.

\end{proof}

\section{Maximal holomorphic maps in one-parameter families}

Given a hyperbolic surface $X$ with at least one funnel, we define enlargements $X_i \supset X$ as follows. We first set $X_0=X$, and for $i \in (0, \infty]$, we define $X_i$ by gluing a copy of the cylinder $S^1 \times [0,i)$ to each ideal boundary component of $X$ in a prescribed way. In other words, we choose once and for all an analytic parametrization of each ideal boundary component of $X$ by $S^1$, and glue the cylinders to $X \cup \partial X$ with these parametrizations. Whenever $i \leq j$, the inclusion map from $S^1 \times [0,i)$ to $S^1 \times [0,j)$ extends to a conformal embedding $X_i \hookrightarrow X_j$ which is equal to the identity on $X$. We call $(X_i)_{i\in I}$ a \emph{one-parameter family}, where $I=[0,\infty]$. 

\begin{figure}[htp]
\includegraphics[scale=.8]{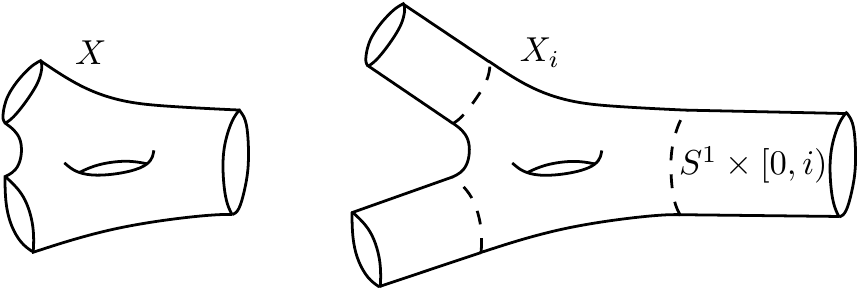}
\caption{A hyperbolic surface $X$ and an enlargement $X_i$. Here $X$ is represented conformally as the interior of a compact surface with boundary.} \label{fig:extension}
\end{figure} 

It is easy to see that $X_i$ moves continuously in Teichm\"uller space for $i \in [0,\infty)$. However, $X_\infty$ lives in a different Teichm\"uller space than $X$. Indeed, the cylinder $S^1 \times [0,\infty)$ is biholomorphic to a punctured disk, so that each funnel end of $X$ becomes a cusp in $X_\infty$.

There is a lot of freedom in the above construction since the parametrizations of the ideal boundary components can be chosen arbitrarily. A concrete choice is to realize $X$ as the interior of a complete bordered hyperbolic surface $\overline X$ with totally geodesic boundary and parametrize the boundary components proportionally to hyperbolic length. Note that the resulting $X_\infty$ is different from the Nielsen extension of $\overline X$, which is $\overline X$ with a funnel glued isometrically to each boundary component. Indeed, every funnel has finite modulus and is thus conformally distinct from the half-infinite cylinder $S^1 \times [0,\infty)$. In the Poincar\'e metric, $X_\infty$ has finite area whereas the Nielsen extension of $\overline X$ has infinite area.

However, the precise construction of the one-parameter family $(X_i)_{i \in I}$ is irrelevant; all that really matters is that there is a conformal embedding $e_{ij}:X_i \hookrightarrow X_j$ homotopic to a homeomorphism whenever $i \leq j$, that the embeddings $e_{ij}$ form a direct system, that the union of the embedded surfaces $\bigcup_{i<j} e_{i j}(X_i)$ equals $X_j$ for every $j>0$, and that all the ends of $X_\infty$ are cusps.     

\begin{lem} \label{lem:maximal}
Let $X$ and $Y$ be hyperbolic surfaces which are neither disks, annuli, or closed surfaces. Suppose that $f:X \to Y$ is a holomorphic map which is homotopic to a homeomorphism, but is not homotopic to a conformal embedding. Given a one-parameter family $(X_i)_{i \in I}$, there exists a largest $i \in [0,\infty)$ such that there exists a holomorphic map $f_i : X_i \to Y$ with $f_i \circ e_{0i}$ homotopic to $f$.
\end{lem} 
\begin{proof}
First observe that no sequence of holomorphic maps $g_n : X \to Y$ homotopic to a fixed homeomorphism $h: X \to Y$ can be compactly divergent. Indeed, $X$ and $Y$ deformation retract onto compact subsets $K$ and $L$, and if $g_n(K)$ is disjoint from $L$, then $g_n(K)$ is contained in an annulus or punctured disk in $Y$. In this case, the homomorphism induced by $g_n$ between fundamental groups is not surjective, so that $g_n$ cannot be homotopic to a homeomorphism. 

Let $i$ be the supremum of the set of indices $j$ for which there exists a holomorphic map $f_j : X_j \to Y$ such that $f_j \circ e_{0j}$ is homotopic to $f$. Using Montel's Theorem, we can extract a holomorphic limit $f_i : X_i \to Y$ from any sequence $(f_{j_n})_{n=1}^\infty$ of such maps for which $j_n \to i$ as $n\to \infty$. Since the domain $X_{j_n}$ is not constant, we actually need to apply Montel's Theorem to larger and larger fixed domains. To be more precise, for every $k<i$ the sequence $(f_{j_n} \circ e_{k j_n})_{n=1}^\infty$ of holomorphic maps from $X_k$ to $Y$ is eventually defined and admits a subsequence converging to a holomorphic map $g_k : X_k \to Y$. To simplify notation, we assume that the original sequence converges. If $k< l < i$, then we claim that $(f_{j_n} \circ e_{l j_n})_{n=1}^\infty$ converges to a holomorphic map $g_l : X_l \to Y$ such that $g_l\circ e_{kl} = g_k$. This is because every subsequence of that sequence has a converging subsequence. The limit $L$ of such a subsequence has to satisfy $L \circ e_{kl} = g_k$ since $(f_{j_n} \circ e_{l j_n})\circ e_{kl} = f_{j_n} \circ e_{k j_n}$. Thus any two limits $L_1$ and $L_2$ of different subsequences have to agree on the open set $e_{kl}(X_k)$ and hence on all of $X_l$ by the identity principle. It follows that the sequence $(f_{j_n} \circ e_{l j_n})_{n=1}^\infty$ converges to the common limit $g_l$ of the converging subsequences. We can therefore define the map $f_i$ on $X_i$ by setting $f_i(x)=g_l\circ e_{li}^{-1}(x)$ for any $l<i$ such that $x \in e_{li}(X_l)$. The restriction $f_i\circ e_{0i}$ of $f_i$ to $X_0=X$ is homotopic to $f$ since homotopy classes of maps between surfaces with finite topology are closed \cite[Corollary 2.7]{couch}.

If $i=\infty$, then $X_i$ has a finite number of ends all of which are cusps. Since cusps can only be mapped holomorphically to cusps (by the Schwarz Lemma), all the ends of $Y$ are also cusps. Furthermore, $f_i$ extends to a holomorphic map $\what{f_i}$ between the compactifications $\what{X_i}$ and $\what{Y}$ where the cusps have been filled in, by the Schwarz Lemma and Riemann's removable singularity theorem. The extension $\what{f_i}$ is a proper holomorphic map between closed Riemann surfaces and hence has some degree $d$. The preimages of the cusps of $Y$ are the cusps of $X_i$, and these two finite sets have the same cardinality. Therefore, if $d>1$ then $\what{f_i}$ has a critical point at some cusp $c$ of $X_i$. In this case, $f_i$ maps small simple loops around $c$ to loops wrapping more than once around $\what{f_i}(c)$, and thus cannot be homotopic to a homeomorphism. It follows that $d=1$, so that $f_i$ is a biholomorphism. But then $f_i \circ e_{0i}$ is a conformal embedding homotopic to $f$, a contradiction. We conclude that $i<\infty$.
\end{proof}

We proceed to prove that the converse of the Schwarz Lemma is false. The proof uses the notion of measured geodesic laminations, for which the reader may consult \cite{Bonahon2}.

\begin{proof}[Proof of Theorem \ref{thm:noconverse}]
Let $f_0: X_0 \to Y$ be a holomorphic map which is homotopic to a homeomorphism but not homotopic to a conformal embedding, where every end of $X_0$ is a funnel. Such a map exists by Theorem \ref{thm:immbutnotemb}. Construct a one-parameter family $(X_j)_{j \in I}$ as above and let $i$ be the largest real number such that there exists a holomorphic map $f_i:X_i \to Y$ whose restriction to $X_0$ is homotopic to $f_0$. The existence of such an $i$ is guaranteed by Lemma \ref{lem:maximal}.

Let $\mathcal{S}$ be the set of homotopy classes of non-trivial simple closed curves in $X_0$. Every $c \in \mathcal{S}$ can be considered as a homotopy class in $X_i$ and in $Y$ via the inclusion $e_{0i}:X_0 \hookrightarrow X_i$ and the map $f_0$ respectively. The length $\ell_{X_i}(c)$ is positive for every $c \in \mathcal{S}$ as all the ends of $X_i$ are funnels. Let $r=\sup_{c \in \mathcal{S}} \ell_Y(c)/\ell_{X_i}(c)$. By the Schwarz Lemma, $r$ is bounded above by $1$. In fact, for every $c \in \mathcal{S}$ we have $\ell_Y(c)/\ell_{X_i}(c) \leq \max_{x \in \gamma} \|d_x f_i\| \leq 1$, where $\gamma \subset X_i$ is the unique geodesic in the homotopy class $c$. This is because $\ell_Y(c)$ is at most the length of $f_i(\gamma)$ which is equal to $\int_\gamma \|d_x f_i\| dx$, hence bounded above by $\ell_{X_i}(c)\cdot \max_{x \in \gamma} \|d_x f_i\|$.

Since all the simple closed geodesics of $X_i$ are contained in the convex core of $X_i$ and since the latter is compact, the inequality $r \leq \|d_x f_i\|$ holds for some $x \in X_i$. Indeed, given a sequence $(c_n)_{n=1}^\infty\subset \mathcal{S}$ such that $\ell_Y(c_n)/\ell_{X_i}(c_n)\to r$, let $x_n \in \gamma_n$ be such that $ \|d_{x_n} f_i\|= \max_{x \in \gamma_n} \|d_x f_i\|$. Then $\ell_Y(c_n)/\ell_{X_i}(c_n)\leq \|d_{x_n} f_i\|$ by the previous paragraph and for the limit $x$ of any converging subsequence of $(x_n)_{n=1}^\infty$ we get the inequality $r \leq \|d_x f_i\|$.

 If $r=1$, then by the case of equality in the Schwarz Lemma, $f_i$ is a covering map. Any covering map homotopic to a homeomorphism is a homeomorphism, so $f_i$ is a biholomorphism. But then the restriction of $f_i$ to $X_0$ is a conformal embedding homotopic to $f_0$, a contradiction. We conclude that $r$ is strictly less than $1$.

Let $\overline{\mathcal{S}}$ be the closure of $\mathcal{S}$ in the space of projective measured laminations on the convex core of $X_0$ (or its double). This space is compact. Moreover, the length functions $\ell_{X_i}$ and $\ell_Y$ extend continuously to measured laminations  \cite{Thurston}. In fact, the function $(c, Z) \mapsto \ell_Z(c)$ defined on the product of the space of measured laminations with the Teichm\"uller space of $X_0$ is continuous \cite{Bonahon}. Since $X_j$ depends continuously on $j \in [0, \infty)$, the function $(c,j) \mapsto \ell_Y(c)/\ell_{X_j}(c)$ is continuous on $\overline{\mathcal{S}}\times [0,\infty)$. From this and the fact that $r$ is strictly less than $1$, it follows that there exists a $j>i$ such that the ratio $r_j= \sup_{c\in \mathcal{S}} \ell_Y(c)/\ell_{X_j}(c) = \max_{c\in \overline{\mathcal{S}}} \ell_Y(c)/\ell_{X_j}(c)$ is still strictly less than $1$. 

By \cite[Proposition 3.1]{Thurston}, the quantity $r_j$ remains unchanged if we extend the supremum to all homotopy classes of closed curves. Therefore, all the homotopy classes of closed curves in $Y$ are shorter than in $X_j$ by a factor $r_j<1$. Moreover, there is no holomorphic map $f_j : X_j \to Y$ whose restriction to $X_0$ is homotopic to $f_0$ by the maximality of $i$.
\end{proof}

\begin{acknowledgements}
I thank Chris Arretines for bringing W. Thurston's paper to my attention and Ara Basmajian for his encouragement. I also thank Dylan Thurston, Joe Adams, and the anonymous referees their careful reading and constructive comments.
\end{acknowledgements}

\bibliographystyle{amsalpha}

\begin{thebibliography}{KPT15}

\bibitem[Bon88]{Bonahon}
F.~Bonahon, \emph{The geometry of {T}eichm\"uller space via geodesic currents},
  Invent. Math. \textbf{92} (1988), 139--162.

\bibitem[Bon01]{Bonahon2}
\bysame, \emph{Geodesic laminations on surfaces}, Laminations and foliations in
  dynamics, geometry and topology (Stony Brook, NY, 1998), Contemp. Math., vol.
  269, Amer. Math. Soc., Providence, RI, 2001, pp.~1--37.

\bibitem[FB15]{couch}
M.~Fortier~Bourque, \emph{The holomorphic couch theorem}, preprint,
  \href{http://arxiv.org/abs/1503.05473}{\tt arXiv:1503.05473}, 2015.

\bibitem[Gen14]{Gendulphe}
M.~Gendulphe, \emph{Trois applications du lemme de {S}chwarz aux surfaces
  hyperboliques}, preprint, \href{http://arxiv.org/abs/1404.4487}{\tt
  arXiv:1404.4487}, 2014.

\bibitem[Iof75]{Ioffe}
M.S. Ioffe, \emph{Extremal quasiconformal embeddings of {R}iemann surfaces},
  Sib. Math. J. \textbf{16} (1975), 398--411.

\bibitem[Ker80]{Kerckhoff}
S.P. Kerckhoff, \emph{The asymptotic geometry of {T}eichm\"uller space},
  Topology \textbf{19} (1980), 23--41.

\bibitem[KPT15]{Dylan}
J.~Kahn, K.M. Pilgrim, and D.P. Thurston, \emph{Conformal surface embeddings
  and extremal length}, preprint, \href{http://arxiv.org/abs/1507.05294}{\tt
  arxiv:1507.05294}, 2015.

\bibitem[Mas97]{Masumoto2}
M.~Masumoto, \emph{Holomorphic and conformal mappings of open {R}iemann
  surfaces of genus one}, Nonlinear Anal. \textbf{30} (1997), 5419--5424.

\bibitem[Mas00]{Masumoto}
\bysame, \emph{Hyperbolic lengths and conformal embeddings of {R}iemann
  surfaces}, Israel J. Math. \textbf{116} (2000), 77--92.

\bibitem[Mil06]{Milnor}
J.~Milnor, \emph{Dynamics in one complex variable}, third ed., Princeton
  University Press, 2006.

\bibitem[Par05]{Parlier}
H.~Parlier, \emph{Lengths of geodesics on {R}iemann surfaces with boundary},
  Ann. Acad. Sci. Fenn. Math. \textbf{30} (2005), 227--236.

\bibitem[PT10]{Papadopoulos}
A.~Papadopoulos and G.~Th\'eret, \emph{Shortening all the simple closed
  geodesics on surfaces with boundary}, Proc. Amer. Math. Soc. \textbf{138}
  (2010), 1775--1784.

\bibitem[Str84]{Strebel}
K.~Strebel, \emph{Quadratic differentials}, Springer, 1984.

\bibitem[Thu86]{Thurston}
W.P. Thurston, \emph{Minimal stretch maps between hyperbolic surfaces},
  preprint, \href{http://arxiv.org/abs/math/9801039}{\tt arXiv:math/9801039},
  1986.

\end{thebibliography}

\end{document}